\documentclass{amsart}
\usepackage{amssymb,amsthm,geometry,graphicx}
\usepackage{enumerate}
\newtheorem{theorem}{Theorem}[section]

\newtheorem{prop}[theorem]{Proposition}
\newtheorem{observation}[theorem]{Observation}
\newtheorem{remark}[theorem]{Remark}

\theoremstyle{definition}

\numberwithin{equation}{section}
\DeclareMathOperator{\Vol}{Vol}
\DeclareMathOperator{\ad}{ad}
\DeclareMathOperator{\tr}{trace}
\DeclareMathOperator{\Ad}{Ad}
\DeclareMathOperator{\Aut}{Aut}
\DeclareMathOperator{\ric}{Ric}

\newcommand{\so}{\mbox{${\mathfrak s \mathfrak o}$}}
\newcommand{\su}{\mbox{${\mathfrak s \mathfrak u}$}}

\newcommand{\af}{\mbox{${\mathfrak a}$}}

\newcommand{\g}{\mbox{${\mathfrak g}$}}
\newcommand{\h}{\mbox{${\mathfrak h}$}}
\newcommand{\kf}{\mbox{${\mathfrak k}$}}

\newcommand{\p}{\mbox{${\mathfrak p}$}}

\newcommand{\tf}{\mbox{${\mathfrak t}$}}
\newcommand{\uf}{\mbox{${\mathfrak u}$}}

\newcommand{\C}{\mbox{${\mathbb C}$}}

\newcommand{\R}{\mbox{${\mathbb R}$}}
\newcommand{\Z}{\mbox{${\mathbb Z}$}}

\begin{document}

\title{On the volume of orbifold quotients of symmetric spaces}
\date{\today}

\author{Ilesanmi Adeboye}
\address{Department of Mathematics and Computer Science, Wesleyan University, Middletown, CT 06459}
\email{iadeboye@wesleyan.edu}

\author{McKenzie Wang}
\address{Department of Mathematics and Statistics, McMaster University, Hamilton, ON, L8S 4K1}
\email{wang@mcmaster.ca}
\thanks{MW was partially supported by NSERC grant OPG0009421}

\author{Guofang Wei}
\address{Department of Mathematics, University of California, Santa Barbara, CA 93106}
\email{wei@math.ucsb.edu}
\thanks{GW was partially supported by NSF DMS grant 1506393}



\begin{abstract} A classic theorem of Kazhdan and Margulis states that for any semisimple Lie group without compact factors, there is a positive lower bound on the covolume of lattices. H. C. Wang's subsequent quantitative analysis showed that the fundamental domain of any lattice contains a ball whose radius depends only on the group itself. A direct consequence is a positive minimum volume for orbifolds modeled on the corresponding symmetric space. However, sharp bounds are known only  for hyperbolic orbifolds of dimensions two and three, and recently for quaternionic hyperbolic orbifolds of all dimensions.

As in \cite{AdeWei12} and \cite{AdeWei14}, this article combines H. C. Wang's radius estimate with an improved upper sectional curvature
bound for a canonical left-invariant metric on a real semisimple Lie group and uses Gunther's  volume comparison theorem to deduce an explicit uniform lower volume bound for arbitrary orbifold quotients of a given irreducible symmetric spaces of non-compact type. The numerical bound for the octonionic hyperbolic plane is the first such bound to be given. For (real) hyperbolic orbifolds of dimension greater than three, the bounds are an improvement over what was previously known.

\end{abstract}

\maketitle
\tableofcontents
\bibliographystyle{plain}

\section{Introduction} A Riemannian manifold $(M,g)$ is a \emph{symmetric space} if for each $x\in M$ there exists an isometry $\phi_x$ of $M$ that fixes $x$, and acts on the tangent space of $M$ at $x$ by minus the identity. Let $G$ be the identity component of the  group of isometries of $M$. The \emph{isotropy subgroup} $K$ at a point $x\in M$ consists of the elements of $G$ that fix $x$. The manifold $M$ is then diffeomorphic to the quotient $G/K$. This article considers simply-connected symmetric spaces of non-compact type that are \emph{irreducible}; that is, not a product of symmetric spaces. In this context, $G$ is a non-compact, simple \emph{real} Lie group and $K$ is a maximal compact subgroup.

Let $\Gamma$ be a discrete subgroup. If the quotient $\Gamma\backslash G$ has a finite measure that is invariant under the action of $G$, then $\Gamma$ is called a \emph{lattice}. Fix a left-invariant Haar measure on $G$. Then the volume of $\Gamma\backslash G$ can be given a value for all lattices $\Gamma$ in a consistent manner. A result, due to Kazhdan and Margulis \cite{KM68}, states that for every \emph{semisimple} Lie group without compact factors, there is a positive lower bound on the covolume of lattices.

For symmetric space $G/K$, and lattice $\Gamma$ of $G$, the double coset space $\Gamma\backslash G/K$ is an \emph{orbifold}. It is a \emph{manifold} when $\Gamma$ is torsion free. If we choose a left $G$-invariant, and right $K$-invariant, inner product on $G$, then we get a fixed left-invariant Haar measure on $G$ as well as a Riemannian submersion $\pi:\Gamma\backslash G \to \Gamma\backslash G/K$ that is an orthogonal projection on the tangent spaces. Hence, $\Vol(\Gamma\backslash G)=\Vol(\Gamma\backslash G/K)\cdot \Vol(K)$. With the volume of compact Lie groups well understood (see e.g. \cite{Mac80}), the study of the volume of orbifold quotients of symmetric spaces is equivalent to the study of covolumes of lattices.

First examples of symmetric spaces of non-compact type are the hyperbolic spaces for each dimension: $\mathbf H^n=SO_0(n,1)/SO(n)$. Complex hyperbolic space $\mathbb C\mathbf H^n$, quaternionic hyperbolic space $\mathbb H\mathbf H^n$, and the octonionic hyperbolic plane $\mathbb O\mathbf H^2$ are the analogs of hyperbolic space defined over each of the remaining normed division algebras. Together, \emph{the hyperbolic spaces} exhaust the category of non-compact symmetric spaces of rank one.

The (real) hyperbolic 2-orbifold of minimum volume was identified by Siegel \cite{Siegel45} in 1945. The corresponding result for 3-orbifolds was proved by Gehring and Martin \cite{GehringMartin09} in 2009. Recently, Emery and Kim determined the quaternionic hyperbolic lattices of minimal covolume for each dimension \cite{EmeryKim}. Fran\c{c}ois Thilmany proved the corresponding result for ${\rm SL}_n(\mathbb R)$ \cite{Thilmany}. These remain the only cases where the minimum orbifold volumes are known. Under various geometric and algebraic constraints (e.g. manifold, cusped, arithmetic) lower bounds for the volume of orbifold quotients of the hyperbolic spaces have been constructed by several authors. Section \ref{history} provides an overview of these results.

Theorem~\ref{MainThm} presents a unified treatment of orbifold volume bounds for non-compact irreducible symmetric spaces. In particular, the octonion hyperbolic case is addressed, as well as spaces of higher rank. For a given semisimple Lie group without compact factors, a computable lower bound on the covolume of lattices, based on the \emph{Margulis constant}, was produced by Gelander in \cite{Gelander}. The techniques of this paper are those of \cite{AdeWei12} and \cite{AdeWei14}, in which hyperbolic and complex hyperbolic volume bounds were derived by combining volume comparison methods in Riemannian geometry with H. C. Wang's quantitative version of the Kazhdan-Margulis theorem. Our results here in particular improve the bounds of \cite{AdeWei12} and \cite{AdeWei14}, and do not require an estimate of the Margulis constant. We have also compared our bound with that of \cite{Gelander} for hyperbolic $4$-orbifolds, using the estimate of the Margulis constant in \cite{Ballmann}, and found that our bound is an improvement.

In \cite{KM68}, Kazhdan and Margulis proved  that every semisimple Lie group without compact factors contains a \emph{Zassenhaus neighborhood} $U$ of the identity; that is \emph{any} discrete subgroup $\Gamma$ in $G$ has a conjugate that intersects $U$ trivially. H.C. Wang undertook a \emph{quantitative} study of Zassenhaus neighborhoods in \cite{Wang69}. Wang associates to each group $G$ a positive real number $R_G$. The value of $R_G$ depends on two constants $C_1$ and $C_2$, which in turn are derived from the root system of, and a choice of inner product on, the Lie algebra of $G$. Using the canonical metric mentioned above, Wang proves that the volume of the fundamental domain of any discrete group $\Gamma<G$ is bounded below by the volume of a ball of radius $R_G/2$. An appendix in \cite{Wang69} lists the values of $C_1$ and $C_2$ for the \emph{classical} non-compact simple Lie groups.

Equipped with the radius $R_G/2$, a lower bound for the volume of the ball in Wang's theorem can be computed from an upper bound on the sectional curvature of the canonical metric using a comparison theorem due to Gunther \cite{Gall}. In this article, the values of $C_1$ and $C_2$ for the \emph{exceptional} non-compact simple Lie groups are computed (see Theorem~\ref{exceptionalgps}) together with those of the classical groups
in a uniform manner. In addition to their role in determining $R_G$, the values of $C_1$ and $C_2$ allow for a uniform estimate of the sectional curvatures of our canonical metric on $G$. As a result, volume bounds for all orbifold quotients of symmetric spaces of non-compact type can be calculated.

  \section{Symmetric Spaces}\label{Sec:SS}

  This section describes the common geometry of symmetric spaces. Further details on the concepts discussed here can be found in Besse \cite{Besse} and Helgason \cite{Helgason}.

 Let $G$ be a Lie group and let $\mathfrak g$ denote the associated Lie algebra. For $X\in\mathfrak g$, \emph{adjoint action} of $X$ is the $\mathfrak g$-endomorphism defined by the Lie bracket; \[\ad X(Y):=[X,Y].\]The \emph{Killing form} on $\mathfrak g$ is the symmetric bilinear form given by \[B(X,Y):=\tr(\ad X \circ \ad Y),\] which is ${\rm Ad}(G)$-invariant.

 By Cartan's criterion, a Lie algebra $\mathfrak g$, and corresponding Lie group $G$, is \emph{semisimple} if and only if the Killing form on $\mathfrak g$ is nondegenerate. A \emph{Cartan decomposition} for a semisimple Lie algebra $\mathfrak g$ is a decomposition \[\mathfrak g = \mathfrak k \oplus \mathfrak p,\] where $[\mathfrak k,\mathfrak k]\subseteq\mathfrak k,[\mathfrak k$, $\mathfrak p]\subseteq\mathfrak p$, and $[\mathfrak p,\mathfrak p]\subseteq\mathfrak k$. Equivalently, one may specify an involutive automorphism $\theta$ of $\g$, in which case
 $\kf$ and $\p$ are respectively the $+1$ and $-1$ eigenspaces of $\theta$. A semisimple Lie algebra may admit more than one Cartan decomposition. In what follows, we are considering semisimple Lie groups together with a \emph{fixed} Cartan decomposition. In this case, the Killing form induces a positive definite inner product (left-invariant Riemannian metric) on $\mathfrak g$ given by  \[ \langle X,Y\rangle = \left\{ \begin{array}{lll} B(X,Y) & \mbox{for } X,Y\in\mathfrak p, \\
-B(X,Y) & \mbox{for } X,Y\in\mathfrak k, \\ 0 & \mbox{otherwise. } \end{array} \right. \]
Then we have $\langle X, Y \rangle = -B(X, \theta(Y))$. This inner product is ${\rm Ad}(K)$-invariant and makes the projection
$\Gamma\backslash G \rightarrow \Gamma\backslash G/K$ into a Riemannian submersion with totally geodesic fibres.
We note here that in Sections ~\ref{Sec:Cs} and ~\ref{Sec:Determine}, we will abuse notation and use $\langle \cdot, \cdot \rangle$ to denote the analogous
metric on $G$ defined by a more convenient scalar multiple of $B$.

The canonical identification of the Lie algebra of a Lie group with its tangent space at the identity, $\mathfrak g\simeq T_eG,$ extends the inner product $\langle \, ,\, \rangle$ to a Riemannian metric on $G$ by left translation. The induced distance function, referred to as \emph{canonical distance}, is denoted by $\rho$.

The objects of our study are \emph{irreducible, simply connected symmetric spaces of non-compact type}.  These spaces are always quotients $G/K$, where $G$ is a real simple non-compact Lie group, and $K=\exp\mathfrak k$ is a maximal compact subgroup of $G$. The restriction of $\langle,\rangle$ to $\mathfrak p\simeq T_{eK}(G/K)$ induces a $G$-invariant Riemannian metric on $G/K$. Such a space is \emph{of type III} if $\mathfrak g^{\mathbb C}=\mathfrak g\otimes\mathbb C$, the complexification of $\mathfrak g$, is simple as a complex Lie algebra, and \emph{of type IV} if not. Note that hyperbolic 3-space $\mathbf H^3$ has both a type III, ${\rm SO}_0 (3, 1)/{\rm SO}(3)$, and a type IV, ${\rm PSL}(2, \mathbb C)/{\rm SO}(3)$, representation.

In what follows, we will make significant use of the correspondence between a symmetric space of non-compact type $G/K$ and its \emph{compact dual} $U/K$. The construction of the Lie group $U$ from $G$ is as follows: Let $G^{\mathbb C}$ be the simply connected complex Lie group that corresponds to $\mathfrak g^{\mathbb C}$. If $\mathfrak g = \mathfrak k \oplus \mathfrak p$, then $\mathfrak u=\mathfrak k\oplus i\mathfrak p$ is a (real) Lie subalgebra of $\mathfrak g^{\mathbb C}$. Let $U$ be the subgroup of $G^{\mathbb C}$ generated by $\mathfrak u$. Since the Killing form $B$ is negative definite on $\mathfrak u$ and $\mathfrak k \subset\mathfrak u$, $U$ is compact and contains $K$. A compact symmetric space is \emph{of type I} or \emph{of type II} if it is dual to, respectively, a space of type III or type IV.

Type II symmetric spaces have the form $(K\times K)/K$. That is, $U=K\times K$ for some compact simple Lie group $K$, and the quotient space is formed with the diagonally embedded subgroup $\Delta K$, also identified with $K$. Tables of symmetric spaces, classified by type, can be found in \cite[p. 201--202]{Besse} and \cite[p. 518]{Helgason}.


The \emph{adjoint representation} of a Lie group $G$, $\Ad: G\rightarrow \Aut(\mathfrak g)$, sends an element of $g\in G$ to the derivative at the identity of the corresponding inner automorphism of $G$. Note that $\Ad_g$ is an isometry of (any multiple) of the Killing form, as well as a Lie algebra automorphism.

The \emph{isotropy representation} of a homogeneous space at a point $x$ is the infinitesimal linear action of $K$ on the tangent space at $x$. Isotropy representations of symmetric spaces are often called \emph{s-representations}. For $G/K$ or $U/K$, \emph{s}-representations can also be described as restrictions of the adjoint action of $G$ (resp. $U$) to $K$ on $\mathfrak p$ (resp. $i\mathfrak p$).

\section{The Constants $C_1$ and $C_2$} \label{Sec:Cs}

The main ingredients in our determination of a lower bound for the volume of $\Gamma \backslash G/K$ are, an explicit positive lower bound for the  size of a fundamental domain of $\Gamma$, and an upper bound for the sectional curvature of $G$. Both quantities are determined by two values, $C_1$ and $C_2$, that in turn depend on the root system of, and a choice of metric on, the Lie algebra $\mathfrak g$ of $G$.

The \emph{length} of $X\in\mathfrak g$ is given by $\|X\|=\langle X,X\rangle^{1/2}$, where the inner product is that defined in Section~\ref{Sec:SS}. The \emph{norm} of an endomorphism $f:\mathfrak g\to\mathfrak g$ is defined by \[N(f)=\sup\left\lbrace\|f(X)\|:X\in\mathfrak g,\|X\|=1\right\rbrace.\]

In \cite{Wang69}, H. C. Wang defined the constants $C_1$ and $C_2$ as follows:
\[C_1 =\sup\left\lbrace N(\ad X)~|~X\in\mathfrak p,\|X\|=1\right\rbrace,\]
\[C_2 =\sup\left\lbrace N(\ad X)~|~X\in\mathfrak k,\|X\|=1\right\rbrace.\]

We give here our refinement of the definitions for $C_1$ and $C_2$. As indicated in the previous section, the group $K$ acts on $\mathfrak p$ and $\mathfrak k$, respectively, by the isotropy or adjoint representation. Recall that for the adjoint representation of the compact Lie group $K$, each element of $\mathfrak k$ is conjugate to an element in a fixed maximal abelian subgalgebra of $\mathfrak k$. Let $X\in\mathfrak k$ and choose $k\in K$ such that $\Ad_k(X)=W$ lies in the preselected maximal abelian subalgebra $\mathfrak a\subset\mathfrak k$. Let $Y$ be any element of $\mathfrak g$. Then,
\begin{align*}
\|\ad X(Y)\|=\|[X,Y]\|&=\|\Ad_k[X,Y]\|\quad\text{($\Ad_k$ is an isometry of the Killing form)}\\
&=\|[\Ad_k(X),\Ad_k(Y)]\|\quad\text{($\Ad_k$ is a Lie algebra automorphism)}\\
&=\|[W,\Ad_k(Y)]\|\\
&=\|\ad W\big(\Ad_k(Y)\big)\|
\end{align*}
Note that as $Y$ runs through all values of $\mathfrak g$, so does $\Ad_k(Y)$.

Less well known is that s-representations, which are examples of \emph{polar actions} (see \cite{Da}), also admit real maximal abelian subalgebras that contain an element from each orbit. Therefore, a similar calculation holds for $X\in\mathfrak p$. Hence we have,

\begin{remark} In determining the constants $C_1$ and $C_2$, it is sufficient to restrict our attention to the norm one elements lying in a fixed but arbitrary maximal abelian subalgebra of $\mathfrak p$ for $C_1$ and of $\mathfrak k$ for $C_2$.
\end{remark}

\begin{theorem}[Wang \cite{Wang69}]\label{RG} Let $G$ be a semisimple Lie group without compact factors, let $e$ be the identity of $G$, let $\rho$ be the canonical distance function, let $R_G$ be the least positive zero of the function \[F(t)=e^{C_1t}-1+2\sin C_2t-\frac{C_1t}{e^{C_1t}-1},\] and let \[\mathcal B_G=\{g\in G~|~\rho(e,g)\leq R_G\}.\] Then for any discrete subgroup $\Gamma$ of $G$, there exists $g\in G$, such that $\mathcal B_G\cap g\Gamma g^{-1}=\{e\}$.
\end{theorem}

It follows that the volume of $G/\Gamma$, for any $\Gamma$, is larger than the volume of the $\rho$-ball in $G$ with radius $R_{G}/2$. 

H. C. Wang calculated that for non-compact, non-exceptional Lie groups either $C_1=C_2$ or $C_2=\sqrt{2}C_1$. Theorem~\ref{Ci-thm} extends this result to the exceptional groups and, moreover, shows that with our choice of metric is either $C_1 = C_2 = \sqrt{2}$ or $C_1 = 1 < \sqrt{2} = C_2$. As H. C. Wang found that $R_GC_1 \approx 0.277$ in the former case and $R_GC_1 \approx 0.228$ in the latter case, we take the value of $R_G/2$ to be $0.098$ or $0.114$.

Let $V(d,k,r)$ denote the volume of a ball of radius $r$ in the complete simply connected Riemannian manifold of dimension $d$ with constant curvature $k$. The following comparison theorem is Theorem 3.101 in \cite{Gall}.

\begin{theorem}[Gunther, see \cite{Gall}]\label{Gun} Let $M$ be a complete Riemannian manifold of dimension $d$. For $m\in M$, let $B_m(r)$ be a ball that does not meet the cut-locus of $m$. Let $k$ be an upper bound for the sectional curvatures of $M$. Then, \[\Vol[B_m(r)]\geq V(d,k,r).\]
\end{theorem}

For fixed, positive values of $d,k$ and $r$, we have by explicit calculation \[V(d,k,r)=\frac{2(\pi/k)^{d/2}}{\Gamma_f(d/2)}\int_0^{rk^{1/2}}\sin^{d-1}\phi\,d\phi.\] Where $\Gamma_f$ represents the \emph{gamma function}. Therefore, a lower bound for the volume of $G/\Gamma$ can be computed by using the dimension of $G$ for $d$, the value $R_G/2$ for $r$, and an upper bound for the sectional curvatures of $G$ for $k$. A bound for sectional curvatures, in terms of $C_1$ and $C_2$, is established in Section~\ref{Sec:Sect}.

\section{Determination of the Constants $C_1$ and $C_2$} \label{Sec:Determine}

In this section we explain how to determine the constants $C_1$ and $C_2$ that are defined and used for curvature estimates in Section~\ref{Sec:Sect}, and then for volume bounds in Section~\ref{Sec:Volume}. The appendix to \cite{Wang69} includes a table of $C_1$ and $C_2$ for the non-compact classical Lie groups. It is therefore necessary to determine these constants for the exceptional Lie groups
as well. In doing so we find that all these constants can be computed in a uniform manner which also explains why their ratio
takes on only two possible values.

By definition, the values of the $C_i$ depend on the Cartan decomposition of the Lie algebra $\mathfrak g$, and on the inner product
$\langle \cdot, \cdot \rangle$ defined in \S 2. Recall that roots are vectors in the dual of the Cartan subalgebra, and we compute their lengths using the bilinear form on the dual of $\mathfrak g^{\mathbb C}$ induced by the Killing form. In this section and the next, we use a renormalized Killing form $B^{\prime}$ to define $\langle \cdot, \cdot \rangle$ where $B^{\prime}$ is the multiple of the Killing form $B$ such that the maximal root of $\mathfrak g^{\mathbb C}$ has length $\sqrt{2}$. This choice reveals a consistency for the values of the $C_i$ that is obscured when the Killing form is used.

\subsection{Type IV}

We discuss first the relatively straight-forward case of type IV irreducible symmetric spaces.

\begin{prop}  \label{type4}
Let $K$ be a compact, simply connected, simple Lie group, and $K^c/ K$ be the corresponding
non-compact symmetric space. Equip $G = K^c$ with the left-invariant metric $\langle \cdot, \cdot \rangle$
described above. Then $C_1 = C_2 = \sqrt{2}$.
\end{prop}

\begin{proof}
In this case the compact form is $U/K=(K \times K)/\Delta K$. The corresponding Cartan decomposition of $\mathfrak u$ is the sum of the diagonal of $\mathfrak k\oplus\mathfrak k$, with the anti-diagonal in $\mathfrak k\oplus\mathfrak k$. Hence the Lie algebra of the non-compact dual is
\begin{align*}
\mathfrak g =\{(X,X)+i(Y,-Y)|X,Y\in\mathfrak k\}&=\{(X+iY,X-iY)|X,Y\in\mathfrak k\}\\
&=\{(Z,\bar{Z})|Z\in\mathfrak k^c\}\\
&\approx\mathfrak k^c\quad\text{(as a \emph{real} Lie algebra)}\\
&=\mathfrak k \oplus i \mathfrak k.
\end{align*}
Hence, the adjoint and isotropy actions of $K$ are equivalent and we have that $C_1 = C_2$.

To compute $C_1$, we already saw that we only need to consider the norms of operators
${\rm ad}(H)$, where $H$ ranges over norm $1$ elements in a fixed maximal abelian subalgebra $\mathfrak a$ in $\mathfrak p=i\mathfrak k$.
But such an algebra is exactly $i$ times a maximal abelian subalgebra  in $\mathfrak k$, which is
the tangent space of a maximal torus in $K$. In other words, $\mathfrak a$ is just a real Cartan
subalgebra in the complex simple Lie algebra $\mathfrak k^c$. For $H \in \mathfrak a$, the action of ${\rm ad}(H)$
on $\mathfrak g = \mathfrak k^c$ can be read off from a Weyl-Chevalley basis of $\mathfrak k^c$.

In fact, starting from the compact group $K$, with inner product $-B^{\prime}|_K$ and a choice of maximal torus and weak order, one can construct a Weyl-Chevalley basis of the form \[ \{ H_1, \cdots, H_{\ell};  X_{\alpha}, X_{-\alpha}, \alpha \in \Delta^{+} \},\]
where $\Delta^{+}$ is the set of positive roots of $K^c$ and $\ell $ is its rank. This basis has the
following properties:
\begin{enumerate}[(a)]
\item The $H_i \in \mathfrak a$ are orthonormal with respect to $\langle \cdot, \cdot \rangle$;

\item $\{\frac{i}{\sqrt{2}} (X_{\alpha} + X_{-\alpha}), \, \frac{-1}{\sqrt{2}}(X_{\alpha} - X_{-\alpha}),\alpha \in \Delta^{+}  \}$ are orthonormal with respect to $\langle \cdot, \cdot \rangle$;

\item  $ [H, X_{\pm \alpha}] = \pm \alpha(H) X_{\pm \alpha}$;

\item $ [X_{\alpha}, X_{-\alpha}] = H_{\alpha}$, where $H_{\alpha}$ is the element in $\mathfrak a$ dual to $\alpha \in \mathfrak a^*$ with respect to $B^{\prime}$.

\end{enumerate}

Let  $Z = X + i Y \in \mathfrak g$, where $X, Y \in \mathfrak k$ and $\|Z\|=1 $. Then,
\begin{align*}
\|[H, Z]\|^2
&=\langle[H, X + i Y],[H, X + i Y]\rangle\\
&=\langle[H, X] + i[H, Y],[H, X] + i[H, Y]\rangle.
\end{align*}

If we expand $X, Y$ in terms of the Weyl-Chevalley basis and apply (a)-(d), and use the fact that
\[|\langle H_{\alpha},H\rangle|\leq\|H_{\alpha}\|\|H\|,\]
we obtain $ \|[H, Z]\|^2 \leq 2$. This shows that $C_1 \leq \sqrt{2}$. However, the upper bound is
realized by setting $H = \frac{1}{\sqrt 2} H_{\mu}$ and $Z=\frac{i}{\sqrt{2}} (X_{\mu} + X_{-\mu})$ where $\mu$ is the maximal
root of $\mathfrak k^c$. This completes the proof of the proposition.
\end{proof}

We now compare the values of $C_i$ obtained above with those in the first three rows of the appendix in \cite{Wang69}. First, note that the Killing form of $U$ induces on $\Delta \mathfrak k$ $\it twice$ the Killing form of $\mathfrak k$.
On the other hand, for the metric $\langle \cdot, \cdot \rangle$ we use on $\mathfrak g \approx \mathfrak k^c$, it is
the negative of the Killing form on $\mathfrak k$. This accounts for the difference
of a factor of $\sqrt{2}$ between our values of $C_i$ and those in the first three rows of the
table in \cite{Wang69}.

Furthermore, to obtain the values of $C_i$ with respect to the Killing form rather than the renormalized Killing form,
one needs to divide by the constant $\sqrt{\alpha_{\mathfrak g}}$ defined by \[B_{\mathfrak g}=\alpha_{\mathfrak g}B^{\prime}_{\mathfrak g}.\]A table for the normalizing constants $\alpha_{\mathfrak g}$ can be found, for example, in \cite[p. 583]{WZ85}.

With the above adjustments, the values we obtained agree with those (in the first three rows of the appendix) in \cite{Wang69}, except
for a possible typographical error there in the value of $C_1$ for the groups of BD type: the factor of 4 in the denominator should be a $2$ instead.

We now discuss the case of irreducible symmetric spaces of type III in two parts.

\subsection{$\mathbf{C_2}$ in Type III}
\begin{prop}  \label{type3-C2}
Let $G/K$ be an irreducible symmetric space of type III. If $G/K={\rm SO}_0 (3, 1)/{\rm SO}(3)$ the constant $C_2$, with respect to the inner product $\langle \cdot, \cdot \rangle$ defined using the renormalized Killing form, is equal to $1$. For all other cases $C_2=\sqrt{2}$.
\end{prop}

\begin{proof}
Assume $G/K\neq{\rm SO}_0 (3, 1)/{\rm SO}(3)$. We fix a maximal abelian subalgebra $\mathfrak t \subset \mathfrak k$ and a weak ordering in $\mathfrak t$. As before, it is
enough to consider the norms of ${\rm ad}(H)$ for norm $1$ elements $H \in \mathfrak t$. Now ${\rm Ad}_K$ acts on $\g = \kf \oplus \p$ as the direct sum of the adjoint and isotropy representations, the latter of which is irreducible over $\R$. If an $\R$-irreducible summand in these
representation split upon complexification, the resulting complex irreducible summands are uniquely characterized by a pair of maximal roots $\mu$ (resp. dominant weights $\lambda$) which have the same norm.  Since a dual pair of irreducible symmetric spaces have isomorphic isotropy representations, it follows that we can work with the type I compact dual spaces $U/K$ as far as $C_2$ is concerned.

For the adjoint representation of $\kf$ we can choose a Weyl-Chevalley basis as in Proposition \ref{type4}.
One sees that the nonzero eigenvalues of ${\rm ad}(H)$ occur among the values $\pm  \alpha(i H)$.
For the $\p$-part we can take a $B^{\prime}$-orthonormal set $\{ w_{\rho}^{\prime}, w_{\rho}^{\prime \prime} \}$
in which $\rho$ ranges over all the positive weights of the isotropy representation. We can then define
the complex weight vectors $v_{\pm \rho} = \frac{1}{\sqrt{2}} (w_{\rho}^{\prime} - (\pm  i w_{\rho}^{\prime \prime}))$
which satisfy $[H, v_{\pm \rho}] = \pm \rho(H) v_{\pm \rho}$. Thus the nonzero eigenvalues of ${\rm ad}(H)$
occur among the values  $ \pm  \rho(i H)$.  It follows from this that if $X\in\mathfrak k$ and $Y\in\mathfrak p$ such that
$|X + Y| = 1$, then $ | [H, X + Y] |^2 \leq {\rm max}(|\mu|^2, |\lambda|^2).$ As in the proof
of Proposition \ref{type4}, both upper bounds are easily seen to be attained when $H = \frac{-i}{|\rho|} H_{\rho}$
for $\rho = \mu $ (resp. $\lambda$).

Note that the norms $|\mu|$ and $|\lambda|$ are computed using the renormalized Killing form of $U$
and not that of $K$. When $U$ and $K$ have the same rank, then the roots of $K$ and the weights of the isotropy
representation are all roots of $U^c$. Hence one of $|\mu|$ or $|\lambda|$ equals $\sqrt{2}$ and we are
done. It therefore remains for us to check if these norms equal $\sqrt{2}$ for $U/K$ with
${\rm rank}\, U > {\rm rank}\, K$.

Recall that the {\em index} of a compact simple Lie subalgebra $\kf$ of a compact simple Lie algebra $\uf$ is the positive integer $[\mathfrak u:\mathfrak k]$ so that $B^{\prime}_{\mathfrak u}= [\mathfrak u:\mathfrak k]B^{\prime}_{\mathfrak k}$. The index allows us to compare norms defined using the renormalized Killing forms of $U$ and $K$. We refer the reader to \cite{Dynkin} for details and p. 584 of \cite{WZ85} for a summary of the pertinent facts that we will use.

Among the type I symmetric spaces $U/K$ with unequal rank and simple $K$; ${\rm SU}(2n)/{\rm Sp}(n), n \geq 2$, ${\rm E}_6/({\rm Sp}(4)/\Z_2) $, and ${\rm E}_6/{\rm F}_4$, the index $[\mathfrak u:\mathfrak k] = 1$. To analyze these cases we will use the standard parametrizations of the root systems of $K$, and knowledge of the isotropy representation (see for example pp. 324-325 of \cite{WZ93}). In the first case, $\mu = 2x_1$ and $\lambda = x_1 + x_2$, so $|\lambda| < |\mu| = \sqrt{2}$. In the second case, $\mu = 2x_1$ while $\lambda = x_1 + x_2 + x_3 + x_4$, and so $|\lambda| = |\mu| = \sqrt{2}.$ In the third case, $\lambda = x_1$ which is shorter than the maximal root $x_1 + x_2$ of ${\rm F}_4$.

Two cases remain: $U/K = {\rm SU}(n)/{\rm SO}(n), n \geq 3$ and
${\rm SO}(2p + 2q +2)/({\rm SO}(2p+1) {\rm SO}(2q+1))$ with $p \geq q \geq 1$ or $p > q = 0$.
(Note that the $n=4$ subcase of the former is the same as the $p=q=1$ subcase of the latter.)
In the first case, the maximal root of ${\rm SO}(n)$ is $x_1 + x_2$, which is shorter than the
dominant weight $2x_1$ of the isotropy representation. However, $[\su(n):\so(n)] =2$, so $2x_1$
has norm $\sqrt{2}$ with respect to the renormalized Killing form of ${\rm SU}(n)$. One needs to
treat the special cases $n=4$ ($\so(4)$ is non-simple) and $n=3$ ($[\su(3):\so(3)] = 4$) separately
but the conclusion is the same.

There are a number of special cases within the second case. If $p \geq q \geq 2$, both simple
factors in $K$ have index $1$. There are two maximal roots $x_1 + x_2$ and $y_1 + y_2$ both
of norm $2$ with respect to the renormalized Killing form. The dominant weight of the isotropy
representation is $x_1 + y_1$ which also has norm $2$. So $C_2 = \sqrt 2$. If $p > q = 0$,
$U/K$ is $S^{2p+1}$, the dual of hyperbolic space, which is treated in \cite{AdeWei12}. With the current point of view,
assuming $p > 1$, we have $[\so(2p+2): \so(2p+1)] =1$. The dominant weight of the isotropy
representation is $x_1$, which is shorter than the maximal root $x_1 + x_2$ of $\so(2p+1)$
of length $\sqrt{2}$. Again we have $C_2 = \sqrt 2$.  For the remaining cases,
either $p$ or $q$ equals $1$. The differences are that $[\so(m):\so(3)] = 2 $ for $m \geq 4$
and the nonzero roots of $\so(3)$ are $\pm 2x$. With these changes one still gets $C_2 = \sqrt{2}$.

Finally, the case of hyperbolic $3$-space, $\mathbf H^3={\rm SO}_0 (3, 1)/{\rm SO}(3)$ with corresponding compact dual $U/K = {\rm SO}(4)/{\rm SO}(3)$ is special because $[\so(4):\so(3)] = 2$ and the maximal root of $\so(3)$ has length only $1$ with respect to the renormalized Killing form of $\so(4)$. Thus $C_2 = 1$ in this special case, as noted in \cite{AdeWei12}. It is also interesting to note
that $\mathbf H^3$ can be written as the type IV symmetric space ${\rm PSL}(2, \C)/{\rm SO}(3)$.
The difference in the values of $C_2$ is due to the difference in the renormalized Killing forms.
\end{proof}

\begin{remark} \label{C2exceptions}{~}
\begin{enumerate}[(i)]
\item The case of hyperbolic $2$-space is also special. The compact form is ${\rm SO}(3)/{\rm SO}(2)$,
so $K$ is abelian and has no nonzero roots. However, ${\rm rank} \, U = {\rm rank}\, K$, so $C_2 = \sqrt{2}$
with respect to the renormalized Killing form of $\so(3)$. Note that in \cite{AdeWei12} the metric used in
defining $C_2$ is twice the renormalized Killing form of $\so(3)$, so the corresponding value of $C_2$ becomes $1$.
\item For complex hyperbolic space, the compact dual is projective space
${\rm SU}(n+1)/{\rm S}({\rm U}(n) {\rm U}(1))$. In \cite{AdeWei14}, the metric used to define $C_2$
is twice the renormalized Killing form of $\su(n+1)$, so $C_2 = 1$ instead.
\end{enumerate}
\end{remark}

\subsection{$\mathbf{C_1}$ in Type III}

So far we have only made use of the structural and representation theory of compact or complex semisimple
Lie groups. We now appeal to elements of the theory of restricted roots for symmetric
spaces. First, we reintroduce some notation and recall basic facts from pp. 257 - 263 of \cite{Helgason}.

Given a type III symmetric space $G/K$ with involution $\theta$ and Cartan decomposition $\g = \kf \oplus \p$,
we fix a maximal abelian subalgebra $\af$ of $\g$ lying in $\p$. As remarked before, we consider
the norms of ${\rm ad}(H)$ for elements $H \in \af$ of norm $1$ with respect to the
metric $\langle \cdot, \cdot \rangle$ defined using the renormalized Killing form of $\g$.

Denote by $\h_{0}$ a fixed extension of $\af$ to a maximal abelian subalgebra of $\g$. The
$\theta$-invariant algebra $\h_{0}$ can be decomposed as $ \tf \oplus \af$ where $ \tf$ is a
maximal abelian subalgebra in $\kf$. The complexification $\h^c$ of $\h_{0}$ is a Cartan subalgebra
of the complex Lie algebra $\g^c$ and $\h_{\R} = i \tf \oplus \af$ is the corresponding real
Cartan subalgebra. The roots of $\g^c$ are regarded as elements of the dual of $\h_{\R}$.
We will also fix a compatible weak ordering of elements in $\h_{\R}$ and $\af$ to get a
consistent notion of positivity for roots and restricted roots.

Let $U/K$ be the dual compact form of $G/K$. Then $\uf = \kf \oplus i\p$ and the maximal
abelian subalgebra $\tf \oplus i \af$ is the Lie algebra of a maximal torus $T$ in $U$. As in the
proof of Proposition \ref{type4}, we choose a Weyl-Chevalley basis of $\g^c$ having the stated
properties there.

Let $\Delta^{+}$ denote the
positive roots of $\g^c$. Upon restriction to $\af \subset \h_{\R}$, some of the roots will
vanish. These are often called the {\it compact roots} and they are characterized by being fixed by
the action of the Cartan involution $\theta$. Those roots which do not vanish identically
may be called {\it noncompact roots} and their restrictions to $\af$, viewed as elements of $\af^*$,
are called {\it restricted roots}. We let $P^{+}$ (resp. $\Sigma^{+}$) denote the
set of positive noncompact (resp. restricted) roots. Their relevance to the determination of
$C_1$ can now be explained.

Choose $H \in \af$ of norm $1$. By Lemma VI.1.2 in \cite{Helgason}, ${\rm ad}(H)$ is self-adjoint
with respect to the metric $\langle \cdot, \cdot \rangle$, which is denoted by $B_{\theta}$ in \cite{Helgason}.
The norm of ${\rm ad}(H)$ as an operator on $\g$ is the absolute value of its largest eigenvalue.
We can analyse its eigenvalues by looking at its complex extension to $\g^c$, where we can use
our Weyl-Chevalley basis. It follows that
\begin{equation} \label{char-C1}
 N({\rm ad}(H)) = {\rm max} \{ |\alpha(H)|: \alpha \in P^{+} \} =
    {\rm max} \{ |\bar{\alpha}(H)|: \bar{\alpha} \in \Sigma^{+} \}.
\end{equation}

The constant $C_1$ is therefore equal to the supremum of these norms over the norm one elements of $\af$.
An upper bound for $C_1$ is clearly $\sqrt{2}$ since we are using the renormalized Killing form.
But $ C_2=\sqrt{2}$ except in the case of hyperbolic $3$-space. Thus $C_1 \leq C_2$ and in the case of hyperbolic $3$-space we have $C_1 = C_2 = 1$, for example by \cite{AdeWei12}.

We can now go through the list of type III symmetric spaces, examine their restricted root systems, and
deduce the values of $C_1$. A convenient source for this purpose is Table VI, pp. 532 - 534 of \cite{Helgason}.
We will exclude the special case of hyperbolic $3$-space in the following discussion.

\begin{observation}\label{Ob} Suppose $\alpha \in P^{+}$ and we know that the dual element
$H_{\alpha} \in \af$. If in addition $\alpha$ is a long root of $\g$ (this is automatic if $\g$ is
of ADE type), then $C_1 = \sqrt{2} = C_2$ since $ \alpha(\frac{H_{\alpha}}{|\alpha|} ) = |\alpha|$.
If all elements of $P^{+}$ are short roots, then $C_1 = 1 < \sqrt{2} = C_2$.
\end{observation}

\begin{theorem} \label{Ci-thm}
Let $G/K$ be a simply connected irreducible symmetric space of non-compact type other than
hyperbolic $3$-space. Equip $G$ with the left-invariant metric induced by the renormalized Killing form.
Let $C_1, C_2$ be the constants defined in Section \ref{Sec:Cs}. Then either $C_1 = C_2 = \sqrt{2}$
or $C_1 = 1 < \sqrt{2} = C_2$. The latter occurs exactly when $G/K$ is one of the following:
\begin{enumerate}
\item  A rank $1$ symmetric space other than $\mathbf H^2$, $\mathbf H^3$, or $\mathbb C\mathbf H^n$,  for $n \geq 2$;
\item ${\rm SU}^*(2n)/{\rm Sp}(n)$, $n \geq 2$;
\item   ${\rm Sp}(m+n)/({\rm Sp}(m) {\rm Sp}(n)),  m \geq n \geq 2$;
\item   or, ${\rm E}_{6(-26)}/{\rm F}_4$.
\end{enumerate}
\end{theorem}

\begin{proof}
The first case to which Observation~\ref{Ob} applies is that of normal real forms. This means that
the maximal abelian subalgebra $\af$ is actually maximal abelian in $\g$, and so can be taken as a
real Cartan subalgebra of $\g^c$. In particular, $\Delta^{+} = P^{+}$ so it always contains the
maximal root. Hence $C_1 = C_2 = \sqrt{2}$. The entries in Table VI of \cite{Helgason} belonging to
this case are:  AI, $r \geq 2$; BI, $3 \leq r = \ell$; CI, $r \geq 2$; DI, $\ell \geq 4$; EI; EV; EVII; FI; and G. Note that ${\rm SO}(3)/{\rm SO}(2)$, corresponding to hyperbolic $2$-space or complex hyperbolic $1$-space, also belongs here as $K$ has no non-zero roots.

Observation~\ref{Ob} is also applicable when $\g$ is of ADE type and there is a restricted
root whose multiplicity is odd, by F3 on p. 530 of \cite{Helgason}. Besides the normal real forms we
have AIII, $r \geq 4$; DI, $r \geq 4$; DIII, $r \geq 3$; EII; EIII; EVI; EVII; and EIX.
In all these cases, again $C_1 = C_2 = \sqrt{2}$. The same conclusion also holds for the case
BI: $ {\rm SO}_0(p, q)/{\rm SO}(p) {\rm SO}(q) $, $p+q$ odd, $p \geq q \geq 2$, since there
are restricted roots which come from long roots whose dual coroot lies in $\af$.

For the case of the Cayley plane ${\rm F}_{4(-20)}/{\rm Spin}(9)$, take as root system for ${\rm F}_4$
\[\{ \pm x_i, 1 \leq i \leq 4; \, \pm x_i \pm x_j, 1 \leq i \neq j \leq 4; \,
     \frac{1}{2}(\pm x_1 \pm x_2 \pm x_3 \pm x_4) \}.\]
A system of positive simple roots is $\{x_2 - x_3, x_3 - x_4, x_4, \frac{1}{2}(x_1 - x_2 - x_3 - x_4) \}$
(see \cite{samelson}, p. 80,  for this choice). In the Satake diagram, a blackened vertex means the
corresponding root restricts to zero for the chosen maximal abelian subalgebra $\af$. So $\af \approx \R$ is
given by $x_2 = x_3 = x_4 = 0$, i.e., $\af^*$ is spanned by the short root $x_1$. The
restricted roots are $\{ x_1, \frac{1}{2}x_2 \}$ with respective multiplicities of $7$ and $8$.
The restriction of the renormalized Killing form on our real Cartan subalgebra is just the Euclidean
inner product. With $H = e_1$ and restricted root $\bar{\alpha}$, we have $\bar{\alpha}(H) = 1$ or $\frac{1}{2}$.
Hence by Equation~(\ref{char-C1}), $C_1 = 1 < \sqrt{2} = C_2$.

For the split rank case of EIV, ${\rm E}_{6(-26)}/{\rm F}_4$, we use the root system for ${\rm E}_6$ given on p. 80 of \cite{samelson}. The positive roots are
\[ \{ x_i - x_j; \, x_i + x_j + x_k; \, x_1 + \cdots + x_6\}, \quad 1 \leq i < j < k \leq 6,\]
with fundamental system
\[\{ x_i - x_{i+1}; \, x_4 + x_5 + x_6 \}, \quad 1 \leq i \leq 5.\]
The renormalized Killing form is $\sum \, x_i^2  + \frac{1}{3} \left( \sum  \, x_i  \right)^2.$
From the Satake diagram we see that $\mathfrak a$ is defined by
\[ x_2 = x_3 = x_4 = x_5, x_4 + x_5 + x_6 = 0, \]
so that $\af = \{ (t_1, t_2, t_2, t_2, t_2, -2t_2) : t_1, t_2 \in \R \}$. For elements of
$\af$ the renormalized Killing form becomes $\frac{4}{3} (t_1^2 + t_1 t_2 + 7 t_2^2)$. The restricted
roots are $\{ t_1 - t_2 ; \, t_1 + 2t_2 ; \, 3t_2  \}$, all with multiplicity $8$.
One then deduces that the maximum value of these linear forms subject to the norm $1$ condition
is $1$ in all cases. Thus $C_1 = 1 < \sqrt{2} = C_2$.

The remaining symmetric spaces of type III that need to be examined are
AII: ${\rm SU}^*(2n)/{\rm Sp}(n)$, $n \geq 2$ and CII: ${\rm Sp}(p, q)/ {\rm Sp}(p) {\rm Sp}(q),\,p \geq q \geq 1 $.  (Recall that the real hyperbolic space case was treated in \cite{AdeWei12}.)

For the case of ${\rm SU}^*(2n)/{\rm Sp}(n)$), $n \geq 2$, we use the usual
parametrization of the real Cartan subalgebra of $\su(2n)$.  The Satake diagram gives
$\mathfrak a$ to be
\[\left\lbrace(t_1, t_1, t_2, t_2, \cdots, t_n, t_n): t_i \in \R,  \sum_{i=1}^{i=n}  \, t_i = 0 \right\rbrace.\]

The positive restricted roots can be taken to be $ t_i - t_j, 1 \leq i < j \leq n$, each with
multiplicity $4$. The dual renormalized Killing form is such that the linear forms $x_i$ are orthonormal.
In particular, an element of $\af$ has norm $1$ if and only if $\sum_i t_i^2 = \frac{1}{2}$. One easily checks
that the maximum value of $t_i - t_j$ on norm $1$ elements of $\af$ is $1$. For $t_1 - t_2$ this is
realised by the element $H = \frac{1}{2}(1, 1, -1, -1, 0, \cdots, 0)$. Hence $C_1 = 1 < \sqrt{2} = C_2$.
Note that the norm $1$ element $\frac{1}{\sqrt{2}}(1, 0, -1, 0, \cdots, 0)$ in the Cartan subalgebra produces
the value $\sqrt{2}$ with $x_1 - x_3$ (which gives the restricted root $t_1 - t_2$), but it does not
lie in $\mathfrak a$.

Similar computations for the CII case, which we leave to the reader, show that $C_1 = 1 < \sqrt{2} = C_2$ as well.
\end{proof}

When $G$ is a classical Lie group and $G/K$ is of type III the above values
of $C_i$ agree with those in the appendix of \cite{Wang69} after adjusting
for the Killing form. For convenient reference, we summarize the results for the exceptional Lie groups developed in this article in the following proposition.

\begin{theorem}\label{exceptionalgps} Let $G/K$ be a simply connected irreducible symmetric
space of non-compact type with $G$ an exceptional Lie group. Then with respect to the metric
induced by the renormalized Killing form, the constants $C_1 = C_2 = \sqrt{2}$, except when
$G/K$  is the Cayley projective plane ${\rm F}_{4(-20)}/{\rm Spin}(9)$ or ${\rm E}_{6(-26)}/{\rm F}_4$,
in which case $C_1 = 1 < \sqrt{2} = C_2$.
\end{theorem}

\section{The Sectional Curvatures of Semi-Simple Lie Groups} \label{Sec:Sect}

Given a semi-simple Lie algebra $\mathfrak g$ with Cartan decomposition $\mathfrak g =\mathfrak k\oplus\mathfrak p$, let $U,V, W \in \mathfrak k$ and $X,Y, Z \in \mathfrak p$ denote left invariant vector fields. The curvature formulas for the canonical metric of a semisimple non-compact Lie group were derived in \cite{AdeWei12}.

\begin{prop}  \label{curvatures}
	\begin{eqnarray}
	R(U,V)W & = & \frac 14 [ [V,U], W],   \label{R-UVW}  \\
	R(X,Y)Z & = & - \frac 74 [ [X,Y],  Z ],    \label{R-XYZ}\\
	R(U,X)Y & = & \frac 14 [[X,U], Y]  - \frac 12 [[Y,U], X] , \label{R-UXY}  \\
	R(X,Y)V & = & \frac 34 [X, [V,Y]] + \frac 34 [Y, [X,V]] = \frac 34 [V, [X,Y]].  \label{R-XYV}
	\end{eqnarray}
	In particular,
	\begin{eqnarray}
	\langle R(U,V)W, X \rangle & = & 0,   \label{R-mix1}\\
	\langle R(X,Y)Z, U\rangle & = &0,  \label{R-mix2}  \\
	\langle R(U,V)V,U\rangle & =  & \frac 14 \|[U,  V]\|^2,  \label{R-UV}\\
	\langle R(X,Y)Y,X \rangle & = & -  \frac 74  \|[X,  Y]\|^2, \label{R-XY}\\
	\langle R(U,X)X,U\rangle & =  & \frac 14   \|[U,  X]\|^2.   \label{R-UX}
	\end{eqnarray}
\end{prop}

\begin{remark} These formulas also apply to any scale of the canonical metric.  \end{remark}

\begin{prop}\label{seccur-G} Let $G$ be a semi-simple Lie group. Let $C_1$ and $C_2$ be the constants defined in Section~\ref{Sec:Cs} with respect to the inner product on $G$ defined in Section~\ref{Sec:Determine}. Let $\alpha=C_2/C_1$.

If $\alpha=1$ the sectional curvatures of $G$ are bounded above by $\frac{49}{52}\times C_1^2$. If $\alpha=\sqrt{2}$ the sectional curvatures of $G$ are bounded above by $1.17259\times C_1^2$.
\end{prop}

\begin{proof} Again with $U,V \in \mathfrak{k}\mbox{ and } X, Y \in \mathfrak{p}$, we have by (\ref{R-mix1}) and (\ref{R-mix2})
	\begin{align*}
		\langle R(X+U,Y+V)Y+V,X+U \rangle  & =  \langle R(X,Y)Y,X\rangle + \langle R(U,V)V,U\rangle + \langle R(U,Y)Y,U\rangle \\
		&   + \langle R(X,V)V,X \rangle  + 2\langle R(X,Y)V, U\rangle + 2\langle R (X,V)Y,U \rangle.
	\end{align*}
	Assume that $\|U+X\|=1,~ \|V+Y\| = 1, \mbox{ and }\langle U+X, V+Y \rangle = 0$. Then \begin{equation}
	\|U\|^2 + \|X\|^2 =1,  \ \ \ \|V\|^2 + \|Y\|^2 =1. \label{unit} \end{equation}
	
	By (\ref{R-UV}), (\ref{R-UX}), and the definitions of $C_1$ and $C_2$,
	\begin{align*}
	\langle R(U,V)V,U\rangle & =   \frac 14 \| \ad U (V)\|^2 \le   \frac 14 C_2^2 \|U\|^2\|V\|^2\\
	\langle R(U,Y)Y,U\rangle & =   \frac 14 \| \ad Y (U)\|^2 \le   \frac 14 C_1^2 \|U\|^2\|Y\|^2\\
	\langle R(X, V)V,X\rangle & =   \frac 14 \| \ad X (V)\|^2 \le   \frac 14 C_1^2 \|X\|^2\|V\|^2.
	\end{align*}
	
	By (\ref{R-XY}),
	\[ \langle R(X,Y)Y,X \rangle  =  -  \frac 74  \|[X,  Y]\|^2 \le 0. \]
		
	By (\ref{R-XYV}),
	\begin{align*}
	\langle R(X,Y)V, U\rangle & =   \frac 34  \langle [U,V], [X,Y] \rangle  \\
	& \le   \frac 34 C_1 C_2 \|X\|\|Y\|\|U\| \|V\|.
	\end{align*}
	
	Similarly,  by (\ref{R-UXY}),
	\[
	\langle R (X,V)Y,U \rangle \le \frac{3}{4}C_1^2 \|X\|\|Y\|\|U\| \|V\|.
	\]
	Combining the above inequalities and using (\ref{unit}), we have
	\begin{align*}
	\lefteqn{ \langle R(X+U,Y+V)Y+V,X+U \rangle } \\
	& \le  \frac 14 \left\{ C_2^2 \|U\|^2\|V\|^2 + C_1^2 \|U\|^2\|Y\|^2 + C_1^2 \|X\|^2\|V\|^2  + 6(C_1 C_2 + C_1^2) \|X\|\|Y\|\|U\| \|V\|   \right\} \\
	& =  \frac 14 C_1^2 \left\{ \|U\|^2+ \|V\|^2  + (\alpha^2 -2)  \|U\|^2\|V\|^2  + 6 (\alpha + 1) (1-\|U\|^2)^{1/2}(1-\|V\|^2)^{1/2}\|U\| \|V\| \right\}.
	\end{align*}
	
	When $\alpha=1$, the polynomial \[a^2+b^2 +(\alpha^2 -2) a^2b^2 +  6 (\alpha + 1) ab (1-a^2)^{1/2}(1-b^2)^{1/2}\] defined on $0 \le a \le 1,\ 0 \le b \le1$ achieves its maximum of $\frac{49}{13}$ at $\left(\sqrt{\frac{7}{13}},\sqrt{\frac{7}{13}}\right)$.
	
	When $\alpha = \sqrt{2}$, the maximum is $4.69036$. \end{proof}
	
	\begin{remark} Compared to the curvature estimates in \cite{AdeWei12} and \cite{AdeWei14}, the above estimate is more general and much improved. It also improves the estimate in \cite{CaoFu}. \end{remark}

\section{Orbifold Volume Bounds} \label{Sec:Volume}

In this section we indicate how to put together the results of the previous sections and derive the desired
explicit volume  bounds. Our calculations make this essentially straight-forward for any orbifold quotient of a symmetric spaces of non-compact type. We will exhibit the bounds in some cases of interest and make some comparisons between our bounds and known bounds in the literature. To do so, we fix a normalization of the metrics used to compute volumes. The standard choice in Riemannian geometry in non-positive curvature situations is to fix the Ricci curvature of an $N$-dimensional manifold to be $-(N-1)$, which is the Ricci curvature of hyperbolic space of constant curvature $-1$ and dimension $N$.

Accordingly, if $G/K$ is a non-compact irreducible symmetric space of dimension $N$, we let $g_0$ to be that multiple of the
Killing form metric on $G/K$ that satisfies \[\ric(g_0)=-(N-1)g_0,\] where $\ric$ denotes the Ricci tensor.
Note however that in general $g_0$ is induced neither by the Killing form nor the renormalized Killing form. The exact relationship
is deduced as follows.

Let $B_G$ denote the Killing form for $G^{\mathbb C}$, the complexification of $G$, and $B^{\prime}_G$ denote 
the renormalized Killing  form, i.e., the multiple of the Killing form such that the maximal root has length $\sqrt 2$. Then since $G$
is simple, \[B_G=\alpha_GB^{\prime}_G,\] where values of $\alpha_G$ for specific $G$ can be found in \cite[p. 583]{WZ85}.
We note that $\alpha_G$ is actually twice the dual Coxeter number for $G$.

Let $g_{_{B_G}}$ and $g_{_{B^{\prime}_G}}$ be the Riemannian metrics on $G/K$ induced by, respectively, $B_G$ and $B^{\prime}_G$. By \cite[Theorem 7.73]{Besse}, we have $\ric(g_{_{B_G}})= -\frac12g_{_{B_G}}$. It follows that \[g_0=\left(\frac{\alpha_G}{2(N-1)}\right)g_{_{B^{\prime}_G}}.\]

Our volume formula now follows from the above discussion and Theorem~\ref{Gun}. We state it formally in the following theorem.

\begin{theorem}\label{MainThm} Let $G/K$ be a non-compact irreducible symmetric space of dimension $N$, let $g_0$ be the metric on $G/K$ induced by the multiple of the Killing form such that the Ricci curvature is $-(N-1)$, and let $\Gamma$ be a lattice of $G$. Then
\[\Vol\left(\Gamma\backslash G/K,g_0\right)\geq\left(\frac{\alpha_G}{2(N-1)}\right)^{N/2}\times\frac{1}{\Vol\left(K,g_{_{B^{\prime}_G}}\right)}\times\frac{2(\pi/k)^{d/2}}{\Gamma_f(d/2)}\times\int_0^{rk^{1/2}}\sin^{d-1}\phi\,d\phi,\]
where, for the Lie group $G$, $\alpha_G$ is twice the dual Coxeter number, $k$ is an upper bound for the sectional curvature, $d$ is the dimension, and $r$ is the radius of a ball that can be embedded in the fundamental domain of $\Gamma$. Furthermore, $g_{_{B^{\prime}_G}}$ denotes the metric on $G/K$ induced by $1/\alpha_G$ times the Killing form, and $\Gamma_f$ denotes the gamma function.
\end{theorem}

Moreover, by Proposition~\ref{seccur-G} and Theorem~\ref{RG}, respectively, values for $k$ and $r$ can be derived from the constants $C_1$ and $C_2$, associated to each Lie group $G$, that were defined in Section~\ref{Sec:Cs}. Theorem~\ref{Ci-thm} proves that there are only two possibilities, $C_1 = C_2 = \sqrt{2}$ or $C_1 = 1$ and $C_2=\sqrt{2}$, and lists them for each group. From the discussion in Section~\ref{Sec:Cs}, and by Proposition~\ref{seccur-G}, we have $r=0.098$ and $k=1.885$ for the first possibility. In the second case, we have $r=0.114$ and $k=1.173$.

We use this formula to calculate explicit lower bounds for orbifold quotients of hyperbolic space, complex hyperbolic space, the octonionic hyperbolic plane, $G_{2(2)}/{\rm SO}(4)$ and $F_{4(4)}/({\rm Sp}(3){\rm Sp}(1)/\Delta\mathbb Z_2)$. We will need the explicit volumes of the compact subgroups $K$, with respect to the renormalized Killing form, of the non-compact simple Lie groups $G$. This can easily be computed from the embeddings of the simple factors of $K$ in $G$, together with the known formulas for the volumes of compact Lie groups (with respect to the Killing form) in the literature. Note that the series of compact classical Lie groups have successive quotients which are spheres. Thus, their volume can be computed by induction. 

\subsection*{Hyperbolic Space: $\mathbf H^n={\rm SO}_0(n,1)/{\rm SO}(n), (n\geq4)$} 
\[N=\dim(\mathbf H^n)=n,\quad d=\dim({\rm SO}_0(n,1))=\frac12n(n+1),\quad \alpha_{{\rm SO}_0(n,1)}=2n-2,\]
\[\text{ and } g_0=\left(\frac{2n-2}{2n-2}\right)g_{_{B^{\prime}_G}}=g_{_{B^{\prime}_G}}.\]

Note that the last fact says that for this special case the Ricci normalized metric coincides with the metric induced by the
renormalized Killing form.

For ${\rm SO}_0(n,1)$, $C_1 = 1$ and $\sqrt{2} = C_2$. Hence, $r=0.114$ and $k=1.173$. Finally, \[\Vol\left({\rm SO}(2p),g_{_{B^{\prime}_G}}\right)=\frac{2^{p-1}(2\pi)^{p^2}}{(2p-2)!(2p-4)!\cdots4!2!}\] and \[\Vol\left({\rm SO}(2p+1),g_{_{B^{\prime}_G}}\right)=\frac{2^{p}(2\pi)^{p^2+p}}{(2p-1)!(2p-3)!\cdots5!3!1!}.\]

%
%

Therefore,
\[
\Vol\left(\mathbf H^{2p}/\Gamma,g_0\right)\geq\frac{\pi^{p/2}(2p-2)!(2p-4)!\cdots2!}{2^{p^2+p-2}\Gamma\left(p^2+\frac{p}{2}\right)(1.17259)^{(p^2+\frac{p}{2})}}\times\int_0^{0.123}\sin^{(2p^2+p-1)}\phi\,d\phi.
\]
Similarly, 
\[
\Vol\left(\mathbf H^{2p+1}/\Gamma,g_0\right)\geq\frac{\pi^{\frac12(p+1)}(2p-1)!(2p-3)!\cdots3!}{2^{p^2+2p-1}\Gamma\left(\frac{2p^2+3p+1}{2}\right)(1.17259)^{(p^2+\frac32p+\frac12)}}\times\int_0^{0.123}\sin^{(2p^2+3p)}\phi\,d\phi.
\]

%
%
%

\begin{remark} For $n=4$, the bound derived in this article is $\Vol\left(\mathbf H^{4}/\Gamma,g_0\right)\ge 5.94845\times10^{-13}$. The corresponding bound using the results of \cite{Gelander} is $118.435\times10^{-52}$. The smallest arithmetic hyperbolic 4-orbifold, which is conjectured to extremal among all hyperbolic 4-orbifolds, has volume approximately $1.8\times10^{-3}$ (see \cite{Belo}).
\end{remark}

\subsection*{Complex Hyperbolic Space: $\mathbb C\mathbf H^n={\rm SU}(n,1)/[{\rm S}\big({\rm U}(n){\rm U}(1)\big)],(n\geq2)$}

\[N=\dim(\mathbb C\mathbf H^n)=2n,\quad d=\dim({\rm SU}(n,1))=n^2+2n,\quad \alpha_{{\rm SU}(n,1)}=2n+2,\]
\[\text{ and } g_0=\left(\frac{2n+2}{4n-2}\right)g_{_{B^{\prime}_G}}=\left(\frac{n+1}{2n-1}\right)g_{_{B^{\prime}_G}}.\]

For ${\rm SU}(n,1)$, $C_1 = C_2 = \sqrt{2}$. Hence, $r=0.098$ and $k=1.885$. We also calculate \[\Vol\left({\rm S}\big({\rm U}(n){\rm U}(1)\big),g_{_{B^{\prime}_G}}\right)=\frac{\sqrt{n+1}\,(2\pi)^{\frac12(n^2+n)}}{(n-1)!\cdots2!}.\] 

We then have,
\[\Vol\left(\mathbb C\mathbf H^n/\Gamma,g_0\right)\geq\left(\frac{n+1}{2n-1}\right)^n\frac{\pi^{\frac{n}{2}}(n-1)!\cdots2!}{\sqrt{n+1}(1.88462)^{\frac12n^2+n}\Gamma\left(\frac12n^2+n\right)2^{\frac{n^2+n-2}{2}}}\times\int_0^{0.134}\sin^{n^2+2n-1}\phi\,d\phi.
\]

%

\begin{remark} For $n=2$, we have $\Vol\left(\mathbb C\mathbf H^2/\Gamma,g_0\right)\geq 7.86511\times10^{-11}$. There is no known, or conjectured, lower bound for complex hyperbolic 2-orbifolds.
\end{remark}

\subsection*{Octonionic Hyperbolic Plane: $\mathbb O\mathbf H^2={\rm F_{4(-20)}}/{\rm Spin}(9)$} 

 \[N=\dim(\mathbb O\mathbf H^2)=16,\quad d=\dim({\rm F_{4(-20)}})=52, \quad \text{ and }\alpha_{_{\rm F_{4(-20)}}}=18.\] Note that since the index $[{\rm F_{4(-20)}}:{\rm Spin}(9)]=1$, the renormalized Killing form of ${\rm F_{4(-20)}}$ restricted to ${\rm Spin}(9)$ is the same as the renormalized Killing form of ${\rm Spin}(9)$. As ${\rm Spin}(9)$ is the double cover of ${\rm SO}(9)$, we have \[\Vol\left({\rm Spin}(9),g_{_{B^{\prime}_G}}\right)=2\Vol\left({\rm SO}(9),g_{_{B^{\prime}_G}}\right)=\frac{2^5(2\pi)^{20}}{7!5!3!}.\]
 
 The values of $C_1$ and $C_2$, and therefore $r$ and $k$, for ${\rm F_{4(-20)}}$ are the same as for ${\rm SO}_0(n,1)$. Therefore,
\begin{align*}
\Vol\left(\mathbb O\mathbf H^2/\Gamma,g_0\right)&\geq \left(\frac{18}{30}\right)^8\times\frac{7!5!3!}{2^5(2\pi)^{20}}\times\frac{2(\pi/1.723)^{26}}{\Gamma_f(26)}\times\int_0^{0.123}\sin^{51}\phi\,d\phi\\
&\ge \frac{7!5!3!}{25!}\frac{3^8}{5^82^{24}}\frac{\pi^6}{(1.17259)^{26}}\int_0^{0.12344}\sin^{51}\phi\,d\phi\\
&\approx 3.46914\times 10^{-76}.
\end{align*} 

\begin{remark} This is the first explicit bound for the volume of octonionic hyperbolic 2-orbifolds.\end{remark}

\subsection*{\boldmath{$G_{2(2)}/{\rm SO}(4)$}}  

 \[N=\dim(G_{2(2)}/{\rm SO}(4))=8,\quad d=\dim(G_{2(2)})=14,\quad\text{and}\quad\alpha_{_{G_{2(2)}}}=8.\]
\[\text{and} \Vol\left({\rm SO}(4),g_{_{B^{\prime}_G}}\right)= 48\sqrt{3}\pi^4.\]

As $G_{2(2)}$ has the same $C_i$ values as ${\rm SU}(n,1)$,
\begin{align*}
\Vol\left(\Gamma\backslash G_{2(2)}/{\rm SO}(4),g_0\right)&\geq \left(\frac{8}{14}\right)^4\times\frac{1}{48\sqrt{3}\pi^4}\times\frac{2(\pi/1.885)^{7}}{\Gamma_f(7)}\times\int_0^{0.134}\sin^{13}\phi\,d\phi\\
&\geq \frac{32\pi^{3}}{7^4(1.885)^76!3\sqrt{3}}\int_0^{0.134}\sin^{13}\phi\,d\phi\\
&\approx 5.427\times10^{-20}
\end{align*} 

\subsection*{\boldmath{$F_{4(4)}/({\rm Sp}(3){\rm Sp}(1)/\Delta\mathbb Z_2)$}}  

 \[N=\dim(F_{4(4)}/[{\rm Sp}(3){\rm Sp}(1)])=28,\quad d=\dim(F_{4(4)})=52,\quad\text{and}\quad\alpha_{_{F_{4(4)}}}=18.\]
\[\text{and} \Vol\left([{\rm Sp}(3){\rm Sp}(1)],g_{_{B^{\prime}_G}}\right)=\frac{2^{21}\pi^{14}}{5!3!}.\]

Again, $F_{4(4)}$ has the same $C_i$ data as ${\rm SU}(n,1)$. Hence,
\begin{align*}
\Vol\left(\Gamma\backslash F_{4(4)}/[{\rm Sp}(3){\rm Sp}(1)],g_0\right)&\geq \left(\frac{18}{54}\right)^{14}\times\frac{5!3!}{2^{21}\pi^{14}}\times\frac{2(\pi/1.885)^{26}}{\Gamma_f(26)}\times\int_0^{0.134}\sin^{51}\phi\,d\phi\\
&\geq \frac{5!\pi^{12}}{3^{13}2^{19}(1.885)^{26}25!}\int_0^{0.134}\sin^{51}\phi\,d\phi\\
&\approx 4.015\times10^{-84}
\end{align*}

\section{History}\label{history}  The \emph{rank one} locally symmetric spaces of non-compact type or \emph{hyperbolic spaces}, denoted by $\mathbb K\mathbf H^n$, where $\mathbb K$ is $\mathbb C$, $\mathbb H$, or $\mathbb O$, are the analogues of $\mathbf H^n$ defined over the complex, quaternion, and octonion algebras. Recall that an orbifold $\mathbb K\mathbf H^n/\Gamma$, is a \emph{manifold} if $\Gamma$ does not contain elliptic elements; \emph{cusped} (non-compact) if $\Gamma$ contains a parabolic element; \emph{closed} (compact) if $\Gamma$ does not contain a parabolic element; and \emph{arithmetic} if $\Gamma$ can be derived by a specific number-theoretic construction (see \cite{Belo}).  Much of the literature on volume bounds for hyperbolic orbifolds involves these subcases and their intersections.

A survey of results of lower bounds of hyperbolic orbifolds, and subsequent improvements, that depend only on dimension and the algebra of definition, is given below. This is followed by an exposition of the cases where sharp bounds are known. Notably, in each case where sharp bounds have been established, or even conjectured, the corresponding orbifold (or manifold) is arithmetic. Furthermore, as shown by Corlette \cite{Corlette} and Gromov and Schoen \cite{GroScho}, for $\mathbb K=\mathbb H,n\geq2$, and $\mathbb K=\mathbb O, n=2$, \emph{all} hyperbolic orbifolds are arithmetic.

\subsection*{Lower Bounds} An explicit lower bound for the volume of a real hyperbolic $n$-manifold, depending only on dimension, was constructed by Martin \cite{Martin89} and improved by Friedland and Hersonsky \cite{FriHer}. An analogous result for orbifolds was given by Adeboye and Wei \cite{AdeWei12}. The results in \cite{AdeWei12} are sharper than those of \cite{Martin89} and \cite{FriHer}, despite the consideration of a larger class of orbit spaces.

Following the method of \cite{Martin89}, Xie, Wang, and Jiang \cite{XieWangJiang} produced a manifold bound in the complex hyperbolic setting. A volume bound for complex hyperbolic $n$-orbifolds was given in \cite{AdeWei14}. Recent results for quaternionic hyperbolic space are a manifold bound due to Cao \cite{Cao16} and an orbifold bound due to Cao and Fu \cite{CaoFu}.

\subsection*{Real Hyperbolic Orbifolds} The smallest hyperbolic 2-orbifold, with area $\pi/21$, is the sphere with 3 cone points of order 2, 3 and 7. The smallest compact and non-compact hyperbolic 2-manifolds are the surface of genus 2 and the once puncture torus. They have area, respectively,  $4\pi$ and $2\pi$. The smallest hyperbolic 3-orbifold, constructed from a tetrahedral reflection group, has volume 0.03905...(Gehring and Martin \cite{GehringMartin09}). The Weeks manifold, obtained by Dehn surgery on the Whitehead link, is the hyperbolic 3-manifold with minimum volume of 0.9427...(Gabai, Meyerhoff and Milley \cite{GMM}). Each of these orbifolds has been shown to be arithmetic (Chinburg and Friedman \cite{Chin1}, Chinburg et al \cite{Chin2}).

The smallest volume non-compact hyperbolic $n$-orbifolds, for dimensions $n\leq9$, were given by Hild \cite{Hild}. These orbifolds have also been shown to be arithmetic. Indeed, the minimal volume compact and non-compact arithemetic hyperbolic $n$-orbifolds, for each dimension $n\geq4$, were identified by Belolipetsky \cite{Belo,Belo2} and Belolipetsky and Emery \cite{BeloEm}.

A comprehensive survey of the results of, and the methodology for, the study of real hyperbolic orbifolds of small volume is \cite{BeloICM}.

\subsection*{Complex Hyperbolic Orbifolds} Hersonsky and Paulin use the Chern-Gauss-Bonnet formulas to prove that the smallest closed complex hyperbolic 2-manifold has volume $8\pi^2$ \cite{HerPau}. Work on the classification of fake projective planes by Prasad and Yeung \cite{PraYeu,PraYeu2} and Cartwright and Steger \cite{CartSte} produced 51 explicit examples. Parker proved that smallest non-compact complex hyperbolic 2-manifold has volume $8\pi^2/3$ and found one such example \cite{Parker}.

In addition, Parker \cite{Parker} identified two non-compact complex hyperbolic 2-orbifolds with volume $\pi^2/27$ and conjectured them to be minimal within that category. Stover proved that these orbifolds are the smallest arithmetic non-compact complex hyperbolic 2-orbifolds \cite{Sto}. In \cite{EmSto}, Emery and Stover determine, for each dimension $n$, the minimum volume cusped arithmetic complex hyperbolic orbifold. However, compact orbifolds with smaller volume are known \cite{Saut}.

\bibliography{RankOne}

\end{document}